\documentclass{amsart}
\usepackage{amsmath}
\usepackage{graphicx}
\usepackage{amsfonts}
\usepackage{amssymb}

\newtheorem{theorem}{Theorem}
\theoremstyle{plain}

\newtheorem{corollary}{Corollary}

\newtheorem{definition}{Definition}
\newtheorem{example}{Example}

\newtheorem{proposition}{Proposition}
\newtheorem{remark}{Remark}

\numberwithin{equation}{section}

\begin{document}
\title[On the  $L_{\infty}$-bialgebra structure]{On the  $L_{\infty}$-bialgebra structure of the rational homotopy groups $\pi_{*}(\Omega \Sigma Y)\otimes \mathbb{Q}$ }
\author{Samson Saneblidze}

\address{A. Razmadze Mathematical Institute,
	     I. Javakhishvili Tbilisi State University 2,
	     Merab Aleksidze II Lane, 0193 Tbilisi,
	     Georgia}

\email{sane@rmi.ge}

\subjclass[2000]{Primary 55P35; Secondary  55S05}
 \keywords{Loop space, homotopy groups, $L_{\infty}$-bialgebra}
\date{}

\begin{abstract}
We introduce the notion of an $L_{\infty}$-bialgebra structure on a vector space. We show that
 the rational homotopy groups $\pi_{*}(\Omega \Sigma Y)\otimes \mathbb{Q}$ admit such a structure for the loop space $\Omega \Sigma Y$ of a suspension $\Sigma Y$  that
 characterizes $Y$ up to rational homotopy equivalence.

\end{abstract}
\maketitle
\begin{center}
\textit{Dedicated to the memory of Nodar Berikashvili}
\end{center}

\section{Introduction}

The homotopy groups $\pi_\ast(\Omega X)$ of  the loops $\Omega X$ on a topological space $X$ have no non-zero coproduct. Nevertheless it may have non-trivial higher order cooperations that form an $L_\infty$-coalgebra structure on
$\pi_\ast(\Omega X).$ The  Samelson product
is compatible with this  structure in a sense  that leads to the notion of an  $L_\infty$-bialgebra.
Let $H_\ast(X)$ denote
the homology with rational coefficients $\mathbb{Q}.$  It admits
an $A_\infty$-coalgebra structure, more precisely, a $C_\infty$-coalgebra structure
dual to the $A_\infty$-algebra and $C_\infty$-algebra structures on the cohomology
$H^\ast(X)$ (cf.  \cite{kade}).
In \cite{SU} the notion of an $A_\infty$-bialgbera is introduced on a vector space $V$ and proved that the loop homology
$H_\ast(\Omega X)$ admits such a structure for a simply connected  space $X.$
The motivation of the paper is Theorem 12.2  in \cite{SU}  asserting  that
the Bott-Samelson bialgebra isomorphism $T^a \tilde{H}_\ast(Y)\approx H_{*}(\Omega \Sigma Y)$ extends to an isomorphism of $A_\infty$-bialgebras where on the left-hand side
the $A_\infty$-bialgebra structure consists of the tensor multiplication and of $A_\infty$-coalgebra structural cooperations extended from $H_\ast(Y).$
There is the (anti)symmetrization functor from the category of $A_\infty$-algebras to the category of $L_{\infty}$-algebras (cf. \cite{lada-sta}, \cite{lada-markl}, \cite{fernandez}), and dually from  the category of $A_\infty$-coalgebras to the category of $L_{\infty}$-coalgebras.
Here we have to modify the above extension rule for the $A_\infty$-coalgebra structure of $H_\ast(Y)$ so that the obtained $A_{\infty}$-coalgebra structural cooperations
of $T^a \tilde{H}_\ast(Y)$ preserve the primitives $PT^a \tilde{H}_\ast(Y)\subset T^a \tilde{H}_\ast(Y) ,$
i.e., the rational homotopy groups $\pi_\ast(\Omega \Sigma Y)\otimes \mathbb{Q}.$ Then the  $L_\infty$-bialgebra structure on $\pi_\ast(\Omega \Sigma Y)\otimes \mathbb{Q}$ is obtained by the symmetrization of the $A_\infty$-coalgebra structure.

Furthermore, the  $L_{\infty} $-bialgebra structure on
 $\pi_\ast(\Omega \Sigma Y)\otimes \mathbb{Q}$ characterizes $Y$ up to rational homotopy equivalence.

The rational homotopy groups admit   $L_{\infty}$-algebra structures, but
for  $\pi_\ast(\Omega \Sigma Y)\otimes \mathbb{Q}$   these structures are degenerated and consists only of the Samelson binary product, since
$\pi_\ast(\Omega \Sigma Y)\otimes \mathbb{Q}$
is a free Lie algebra for an arbitrary  $Y.$
In general,
it may have a sense to establish compatibility relation between an $L_{\infty}$-algebra and an $L_{\infty}$-coalgebra structures on a vector space $V$ by means of the
symmetrization of the aforementioned $A_\infty$-bialgebra structure on $V.$

\section{$A_\infty$-coalgebras }
A differential graded coalgebra (dgc) $(C_\ast,d, \Delta: C_\ast\rightarrow C_\ast\otimes C_\ast )$ is non-negatively graded. It is \emph{connected} if $C_0=\mathbb{Q}.$ A dgc  may be  coassociative  or not.
The \emph{reduced} coalgebra $\tilde{C_\ast}$  is defined by $\tilde{C_\ast}=C_{>0},$  and
 $PC\subset C$ denotes the vector subspace of the primitives, $PC=\{c\in C\mid \Delta(c)=1\otimes c+c\otimes 1\}$  for $1\in C_0.$ We assume dgc's are connected and of finite types unless otherwise is stated explicitly.
The cellular chains  $(C_\ast(K_n),d,\Delta_K)$ of the associahedron $K_n,\,n\geq 2,$
is a non-connected, non-coassociative dgc with the coproduct
 \begin{equation}\label{delta}
   \Delta_K: C_*(K_n)\rightarrow C_*(K_n)\otimes C_*(K_n)
    \end{equation}
defined as follows (cf. \cite{SU2}). Recall the partial (Tamari) ordering
on the vertices  of $K_n$ defined by  $u\leq v$  if there
exists an oriented edge-path  from $u$ to $v$ in $K_n.$
 Denote the minimal and maximal vertices of a cell $a$ of $K_n$
by $\min a$ and $\max a$ respectively, and extend the partial ordering
on the cells  of $K_n$ by $a\leq b$ if $\max a\leq \min b.$
 Then
\begin{equation}\label{diagonal}
\Delta_{K}\left(  e\right)  =\sum_{\substack{|a|+|b|=|e|\\ a\leq b}}
sgn(a,b)\,a\otimes b,\ \ \ a\times b\subset e\times e,
\end{equation}
where
$e\subset K_n$ is a cell of $K_n$ in which the top cell is denoted by $e^{n-2}$ for $n\geq 2.$

Given a dg vector space $(C,d),$ an $A_{\infty}$-coalgebra structure  \[(C,d,\{\psi_n: C\rightarrow C^{\otimes n}\}_{n\geq 2})\ \ \text{on}\ \ C\]
 is defined by a chain (operadic) map
\[ \psi : C_\ast(K_n)\rightarrow Hom(C,C^{\otimes n})\ \   \text{with}\ \  \psi(e^{n-2})=\psi_n,\ \ \text{a map of degree}\ \  n-2. \]
In particular, $(C,d, \Delta:=\psi_2)$ is a dgc. Denoting  $\psi_1:=d,$  for each $n\geq 1 $ the cooperations $\psi_n$
satisfy the following  quadratic relations in $Hom(C,C^{\otimes n})$
\begin{equation}\label{Acoalgebra}
\sum_{\substack{ 0\leq k\leq n-1  \\ 0\leq i\leq n-k-1}}
(-1)^{k\left( n+i+1\right) }    \left(id^{\,\otimes i} \otimes \psi_{k+1}\otimes id^{\,\otimes n-k-1-i}\right)\circ  \psi_{n-k}=0.
\end{equation}
\textbf{A $C_\infty$-coalgebra}  $(C,d, \{\psi_r\}_{r\geq 2})$ consists of the data similarly to that of   an $A_\infty$-coalgebra
 but  specified by the condition that
 each dual   operation  $\psi_r^\ast:(C^\ast)^{\otimes r} \rightarrow C^\ast$
 vanishes  on the decomposables under the  shuffle product on $TC^\ast=\underset{k\geq 1}{\bigoplus} (C^\ast)^{\otimes k}.$

Given two  $A_\infty$-coalgebras $(A, d, \{\psi^A_r\}_{r\geq 2})$ and
$(B, d, \{\psi^B_r\}_{r\geq 2}),$
the definition of their  tensor product
$(A\otimes B,d_\otimes,\{\Psi_r\}_{r\geq 2})$ relies on the coproduct
(\ref{delta}) as follows.
Let  the map
 \[\chi: Hom(A,A^{\otimes r})\otimes Hom(B,B^{\otimes r})\rightarrow
Hom(A\otimes B, (A\otimes B)^{\otimes r})\]
be defined by the composition
$\chi:=\sigma^\ast_{r,2}\circ\iota,$ where
\[\iota: Hom(A,A^{\otimes r})\otimes Hom(B,B^{\otimes r})\rightarrow
Hom(A\otimes B,A^{\otimes r}\otimes B^{\otimes r})\]
is the standard map and
 \[ \sigma^\ast_{r,2} :Hom(A\otimes B, A^{\otimes r}\otimes B^{\otimes r})\rightarrow Hom(A\otimes B,  (A\otimes B)^{\otimes r})  \]
is induced by the standard permutation
$\sigma_{r,2}:A^{\otimes r}\otimes B^{\otimes r}\rightarrow
(A\otimes B)^{\otimes r}.$
For each $r\geq 2,$  the tensor cooperation
\[\Psi_r: A\otimes B\rightarrow (A\otimes B)^{\otimes r}\]
 satisfying (\ref{Acoalgebra})  is  given by
 \[{\Psi_{r}}=\chi \circ (\psi^A\otimes\psi^B)\circ \Delta_{K}(e^{r-2}).\]

Let now $C$ be a dg algebra $(C,d,\mu)$ and an $A_\infty$-coalgebra $(C,d,\{\psi_r\}_{r\geq2})$ simultaneously, and   $(C\otimes C,d_\otimes,\{\Psi_r\}_{r\geq 2})$  be the tensor $A_\infty$-coalgebra. Then $(C,d,\mu,\{\psi_r\}_{r\geq2})$ is  \textbf{an $A_\infty$-bialgebra} if
the following diagram
\begin{equation*}
\begin{array}{cccccc}
C\otimes C  & \xrightarrow{\Psi_r} & (C\otimes C)^{\otimes r} \\
 \hspace{-0.1in} \mu\,{\downarrow} & &\hspace{0.1in}{\downarrow}\,\, \mu^{\otimes r} \\
 C  &  \xrightarrow{\psi_r} &   C^{\otimes r}
  \end{array}
\end{equation*}
commutes, i.e.,
the following equation holds
\begin{equation}\label{bialg}
{\psi_{r}}\circ \mu=\mu^{\otimes r}\circ {\Psi_{r}},\ \
r\geq2.
\end{equation}
In small dimensions (\ref{bialg}) reads :
\[
\begin{array}{llll}
\psi_2\circ \mu &=& (\mu\otimes \mu) \circ \chi \circ (\psi_2\otimes \psi_2)    \vspace{1mm}\\

\psi_3\circ \mu &=&  \mu^{\otimes 3}\circ\chi\circ
\left( (\psi_2\otimes 1)\psi_2\otimes \psi_3  +     \psi_3\otimes (1\otimes \psi_2)\psi_2 \right)\vspace{1mm}\\

\psi_4\circ \mu &=&   \mu^{\otimes 4}\circ\chi\circ\\
&& ( (\psi_2\otimes 1\otimes 1)(\psi_2\otimes 1)\psi_2\otimes \psi_4
+
\psi_4\otimes  (1\otimes 1\otimes \psi_2)(1\otimes \psi_2)\psi_2  \vspace{1mm}\\
&& +\,
(\psi_3\otimes 1)\psi_2\otimes ((1\otimes \psi_2\otimes 1)\psi_3 + (1\otimes \psi_3)\psi_2)  \vspace{1mm}\\

&&+\,(1\otimes \psi_2\otimes 1)\psi_3\otimes (1\otimes \psi_3)\psi_2 \vspace{1mm} \\
&&-\,
(\psi_2\otimes 1\otimes 1)\psi_3\otimes (1\otimes 1\otimes \psi_2)\psi_3).
\end{array}
\]
Equation (\ref{bialg}) rewrite  as follows. Denote   $ xy:=\mu(x,y)$  and
\[(x_1\otimes \cdots \otimes x_r)\cdot (y_1\otimes \cdots \otimes y_r):=
 x_1y_1\otimes \cdots \otimes x_r y_r   ,\]
  and for cells $a,b$  of $K_r$
 and  for $x,y\in C$
\begin{equation}\label{tformula}
  \psi_r(xy)=\sum_{\substack{|a|+|b|=r-2\\ a\leq b}}
  sgn(a,b)\,
  \psi(a)(x)\cdot \psi(b)(y).
\end{equation}

In particular,   $(C,d,\mu,\psi_2)$  is a bialgebra (Hopf algebra). Furthermore,
given an $A_\infty$-coalgebra
$(C, d,\{\Delta_r :C\rightarrow C^{\otimes r}\}_{r\geq2}),$ consider the tensor algebra $(T^a(\tilde C),d,\mu).$ Use the freeness
of $T^a(\tilde C),$ and by induction on the tensor wordlength apply to
formula (\ref{tformula})  to extend each cooperation $\Delta_r$ to the cooperation
$\psi_r\!:T^a(\tilde C)\rightarrow   T^a(\tilde C)^{\otimes r},$  and, hence, to obtain the  $A_{\infty}$-bialgebra $(T^a(\tilde C), d, \mu, \{\psi_r\}_{r\geq2}).$

\begin{remark}
The coproduct $\Delta_K$ in (\ref{delta})  is not coassociative, so that we fix the left most association by iterative application of (\ref{tformula}).
\end{remark}

The homology $H_\ast(\Omega X)$ admits an $A_\infty$-bialgebra structure \cite{SU}. However,  for a suspension $X=\Sigma Y$ this structure is specified
by  the fact that
the Bott-Samelson isomorphism $T^a \tilde{H}_\ast(Y)\approx H_{*}(\Omega \Sigma Y)$  induced by the inclusion $Y\hookrightarrow \Omega \Sigma Y$   extends to that of the $A_\infty$-bialgebras. In particular,
 the $A_\infty$-algebra substructure on $H_{*}(\Omega \Sigma Y)$ reduces to the  loop (Pontryiagin) multiplication  because
  $H_{*}(\Omega \Sigma Y)$ is a free algebra.

However, $\psi_r$ given by (\ref{tformula}) does not preserve  the primitives $PT^a(\tilde C)\subset
T^a(\tilde C) ,$ so that we have to modify (\ref{tformula})
 as follows. Given a dg coalgebra $(C,d,\Delta),$
think  of  the primitive subcoalgebra $(PC,d, \Delta)$ as a degenerated $A_{\infty}$-coalgebra, and consider two tensor $A_\infty$-coalgebras
 \[\left(A\otimes B, \{^P\!\Psi_r\}_{r\geq 2}\right)= (PA,d_A,\Delta)\otimes (B,\{\psi^B_r\}_{r\geq2})\]
   and
 \[\left(A\otimes B, \{\Psi_r^P\}_{r\geq 2}\right)=
 (A,\{\psi^A_r\}_{r\geq2})\otimes (PB,d_B,\Delta).
 \]
In fact $^P\!\Psi_r$ and  $\Psi^{P}_r,$ referred to as \emph{primitive tensor cooperations}, are of the form
\[
^P\!\Psi_r = \Delta ^{(r-1)}\otimes \psi_r \ \ \text{ and}\ \ \Psi^{P}_r = \psi_r \otimes \Delta ^{(r-1)}
\]
respectively,
where $ \Delta ^{(r-1)}:PC\rightarrow PC^{\otimes r}$ denotes $(r-1)$-iteration of $\Delta=\Delta ^1$ for $r\geq 2.$

Given an  $A_\infty$-coalgebra $(C,d, \{\Delta_r\}_{r\geq 2}),$ for each $r\geq2$ define the cooperation
\[
\varrho_r: T^a(\tilde C)\rightarrow T^a(\tilde C)^{\otimes r}
 \]
 with  $\varrho_r|_C=\Delta_r$ as follows.
 Set $C=A=B$ above, and
form the sum
\begin{equation}\label{sum}
  \varrho_r|_{C^{\otimes 2}}:=\,^P\!\Psi_r + \Psi^{P}_r \ \ \text{on} \ \  C\otimes C.
  \end{equation}
    Then set
$A=C\otimes C$  with $\psi^A_r= P\Psi_r|_{C^{\otimes 2}}$   and form the sum
 \[
 \varrho_r|_{C^{\otimes 3}}:=\,^P\!\Psi_r + \Psi^{P}_r \ \
  \text{on} \ \  C^{\otimes 3},\]
   and so on.
 Obviously, the cooperations $ \varrho_r$  are related with the product on $T^a(\tilde C)$ by the following formula
 \begin{equation}\label{tformula2}
  \varrho_r(xy)= \sum_{1\leq i\leq r}
   y_1\otimes \cdots \otimes x y_i\otimes \cdots \otimes y_r+
    x_1\otimes \cdots \otimes x_i y\otimes \cdots \otimes x_r,
\end{equation}
  where $\varrho_r(x):=x_1\otimes\cdots \otimes x_r$ (the Sweedler type notation).
The following proposition is immediate
\begin{proposition}\label{primitives}
The cooperations  $\varrho_r$  given by (\ref{tformula2})  preserve the vector subspace $ PT(\tilde C)\subset T^a(\tilde C)$
  \[  \varrho_r: PT(\tilde C)\rightarrow   PT(\tilde C)^{\otimes r}.
  \]
  \end{proposition}

\begin{definition} An $A_\infty$-coalgebra $(C, d,\{\Delta_r\}_{r\geq 2}),$ is \textbf{primitive} if the cooperations
\[             \varrho_r:  T^a(\tilde C)\rightarrow    T^a(\tilde C)^{\otimes r} ,\ \ r\geq 2,      \]
satisfy (\ref{Acoalgebra})  and, hence,  form $A_{\infty}$-coalgebra structure on $(T^a(\tilde C),d).$
\end{definition}

Denoting  $[x,y]=xy-(-1)^{|x||y|}yx,$  and taking into account (\ref{tformula2})  we immediately obtain
\begin{proposition}\label{Aprimitives}
 A primitive $A_\infty$-coalgebra $(C, d,\{\Delta_r\}_{r\geq 2})$ induces
 the $A_{\infty}$-coalgebra structure on $(PT^a(\tilde C),d)$
 satisfying the equality
 \begin{equation}\label{brformula}
  \varrho_r[x,y]=\sum_{1\leq i\leq r}
   y_1\otimes \cdots \otimes [x, y_i]\otimes \cdots \otimes y_r+
    x_1\otimes \cdots \otimes [x_i, y]\otimes \cdots \otimes x_r.
\end{equation}
  \end{proposition}

When  $(C=PC,d,\Delta_2)$ is a  primitive dgc,
compare the two induced $A_\infty$-coalgebra structures
\[(T^a(\tilde C), d,\{\psi_r\}_{r\geq 2} )\ \   \text{and}\ \
\ \   (T^a(\tilde C), d, \{\varrho_r\}_{r\geq 2} ) \ \
 \text{on} \ \ T^a(\tilde C)\]
to deduce that
\begin{equation}\label{three}
  2\psi_2=  \varrho_2    \ \ \text{and}   \   \      \psi_3=  \varrho_3,
       \end{equation}
  while on the decomposables
\[  \psi_r=\varrho_r+ \bar \psi_r\ \   \text{for} \  \  r\geq 4, \]
where $\bar \psi_r$ is the non-primitive summand component of  $\psi_r$ in (\ref{tformula}).
In particular, (\ref{three}) implies
\begin{proposition}
  An  $A_\infty$-coalgebra   of the form $(C=PC, d,\{\Delta_2, \Delta_3, 0,...\})$
  is primitive.
\end{proposition}

\section{$L_\infty$-coalgebras }

The notion of an $L_\infty$-coalgebra is dual to that of an $L_\infty$-algebra
\cite{lada-sta}, \cite{lada-markl}.
Let \[S(n):C^{\otimes n}\rightarrow C^{\otimes n}\] be a map defined for $a_1\otimes \cdots \otimes a_n\in C^{\otimes n} $ by
\[ S(n)(a_1\otimes \cdots \otimes a_n)=\underset{\sigma\in S_n}{\Sigma}   sgn(\sigma)\varepsilon(\sigma)\,a_{\sigma(1)}\otimes \cdots \otimes a_{\sigma(n)},\]
where $sgn(\sigma)$ is the standard sign of a permutation $\sigma$ and
$\varepsilon(\sigma)$ is determined by the Koszul sign rule.
Let $S_{i,n-i}\subset S_n$ denote the subset of $(i,n-i)$-shuffles with $S_{n,0}=1\in S_n,$ and let
$S(i,n-i):C^{\otimes n}\rightarrow C^{\otimes n}$ be a map defined for $a_1\otimes \cdots \otimes a_n\in C^{\otimes n} $ by
\[ S(i,n-i)(a_1\otimes \cdots \otimes a_n)=\underset{\sigma\in S_{i,n-i}}{\Sigma}   sgn(\sigma)\varepsilon(\sigma)\,a_{\sigma(1)}\otimes \cdots \otimes a_{\sigma(n)}.\]

\textbf{An $L_\infty$-coalgebra} is a dg vector space $(L,d)$ together with
linear maps \[\{\ell^r:L\rightarrow L^{\otimes r}\}_{r\geq 1}\ \  \text{of degree}\ \ r-2\ \  \text{with}\  \  \ell^1:=d\]
 such that

(i) $\ell^r=S(r)\circ\ell^r,\, r\geq 1;$

(ii)$ \underset{1\leq i\leq n}{\sum} (-1)^{i(n-i)} S(i,n-i) \circ
(\ell^i\otimes 1^{\otimes n-i}) \circ \ell^{1+n-i}=0.$

In particular, $(L, \ell^2)$ is a (graded) Lie coalgebra,
 when $d=0,$ or
$( L,  d, \ell^2 )$ is a dg Lie coalgebra, when $\ell_3=0.$
Denote $\ell^r(x):=x_1\otimes\cdots \otimes x_r$ (the Sweedler type notation),
and
 $\ell_2(x,y):=[x,y].$
\begin{definition}
  Let $L$ be a dg Lie algebra $(L,d,\ell_2)$ and an $L_{\infty}$-coalgebra
 $(L, d, \{\ell^r\}_{r\geq 2})$ simultaneously.
  Then $(L, d,\ell_2, \{\ell^r\}_{r\geq 2})$ is
\textbf{an $L_{\infty}$-bialgebra }
if for each $r\geq 2$
\begin{equation}\label{lformula}
  \ell^r[x,y]=\sum_{1\leq i\leq r}
   y_1\otimes \cdots \otimes [x, y_i]\otimes \cdots \otimes y_r+
    x_1\otimes \cdots \otimes [x_i, y]\otimes \cdots \otimes x_r.
\end{equation}
\end{definition}
In particular,  $(L, \ell_2, \ell^2)$ is a (graded) Lie bialgebra,
 when $d=0,$ or
$( L,  d, \ell_2, \ell^2 )$ is a dg Lie bialgebra, when $\ell_3=0$
in a sense \cite{etin-schiff}.
A motivated example of an $L_\infty$-bialgebra is given by Theorem \ref{PL} below.

\subsection{Symmetrization}

Given an $A_\infty$-coalgebra   $\left( C,d, \{\psi_r:
C\rightarrow C^{\otimes r} \}_{r\geq2}\right),$ there is
the associated $L_\infty$-coalgebra   $(L,d, \{\ell^r:L\rightarrow L^{\otimes r}\}_{r\geq2})$
where $(L,d)=(C,d)$ and for each  $r\geq 2$  the structural cooperation $\ell^r: L\rightarrow L^{\otimes r}$ is
obtained by the symmetrization
\[  \ell^r=\psi_r^{sym}\ \   \text{for}   \ \   \psi_r^{sym}:= S(r) \circ \psi_r. \]

\begin{theorem}\label{PL} If the structural cooperations of an  $A_\infty$-coalgebra $(C, d,  \{\Delta_r\}_{r\geq2})$
restrict to $\Delta_r:PC\rightarrow PC^{\otimes r}$ for all $r$ and form a primitive
$A_{\infty}$-coalgebra structure on $PC,$ then
the free  Lie algebra $(L(PC), \ell_2)$ admits a canonical $L_\infty$-bialgebra structure $(L(PC), d,\ell_2, \{\ell^r\}_{r\geq 2} )$
with $\ell^2=0.$
\end{theorem}
\begin{proof} First, recall that $PT^a(\tilde C)=L(PC),$ the free Lie algebra
generated by $PC.$ Apply to Proposition \ref{Aprimitives}  and  obtain
$\ell^r: L(PC)\rightarrow L(PC)^{\otimes r}$  as $\ell^r=\varrho_{r}^{sym}$ on $L(PC).$ Then  (\ref{brformula}) implies (\ref{lformula}). Since $\Delta_2: PC \rightarrow PC\otimes PC$ is cocommutative, $\ell^2=0.$

\end{proof}
Denoting $L:=L(PC),$
 equality (\ref{lformula}) is equivalent to the following commutative diagram
\begin{equation}\label{biinfinity}
\begin{array}{cccccc}
L\otimes L  & \xrightarrow{\Phi_r^{sym}} & (L\otimes L)^{\otimes r} \\
 \hspace{-0.1in} \ell_2\,{\downarrow} & &\hspace{0.1in}{\downarrow}\,\, \bar \ell_2 \\
 L  &  \xrightarrow{\ell^r} &   L^{\otimes r}
  \end{array}
\end{equation}
where $\Phi_r:=\,^P\!\Psi_r + \Psi^{P}_r$ in which  the right-hand side is defined
by (\ref{sum}) for $C=L,\,   \psi_r=\varrho_r, $
and
$\bar \ell_2:=\underset{1\leq i\leq r} {\sum}  \mu^{\otimes i-1}\otimes \ell_2 \otimes  \mu^{\otimes r-i-1}.$ In fact  $\bar \ell_2$
 consists of only one non-trivial monomial because of
   $[1, -]=[-\,,1]=0$ for $1\in T^a(\tilde C).$

Let $C= H_\ast(Y).$ Then $H_\ast(Y)$ admits an $A_{\infty}$-coalgebra structure, or more precisely, a $C_\infty$-coalgebra structure \cite{kade}.
Taking into account  the Milnor-Moore theorem   we have
\[\pi_\ast(\Omega \Sigma Y)\otimes \mathbb{Q}=   P H_\ast (\Omega \Sigma Y)=PT^a(\tilde H_\ast(Y))=L(PH_\ast(Y)),\]
and then Theorem \ref{PL}
implies
\begin{theorem}
  If the cooperations $\Delta_r: H_\ast(Y)\rightarrow H_\ast(Y)^{\otimes r}$ for $r\geq 3$ preserve the primitives $P H_\ast(Y)$ and form a primitive $A_{\infty}$-
  coalgebra structure on $P H_\ast(Y),$  then
  the rational homotopy groups $ (\pi_\ast(\Omega \Sigma Y)\otimes \mathbb{Q} ,\ell_2)$
  being the Lie algebra with the Samelson product $\ell_2$  admit a canonical  $L_\infty$-bialgebra structure $ (\pi_\ast(\Omega \Sigma Y)\otimes \mathbb{Q} ,  \ell_2, \{\ell^r\}_{r\geq 2})$   with $\ell^2=0.$
\end{theorem}

\begin{corollary} For $PH_\ast(Y)=H_\ast(Y)$ and $(H_\ast(Y),\{\Delta_r\}_{r\geq 2})$ to be primitive,
   the rational homotopy groups
   $ \pi_\ast(\Omega \Sigma Y)\otimes \mathbb{Q} $
    admit a canonical  $L_\infty$-bialgebra structure.
\end{corollary}

In general, the rational homotopy groups admit
an $L_{\infty}$-algebra structure   with the higher order operations $\ell_r$ rather than $\ell_2$ (cf. \cite{fernandez}), but
in the case of   $\pi_\ast(\Omega \Sigma Y)\otimes \mathbb{Q}$  it
 reduces to the Samelson bracket $\ell_2$ because $(\pi_\ast(\Omega \Sigma Y)\otimes \mathbb{Q},\ell_2)$ is a free Lie algebra. Furthermore, note that
although  the Lie coalgebra structure of $ \pi_\ast(\Omega \Sigma Y)\otimes \mathbb{Q} $  is \emph{abelian}, the higher order cooperations $\ell^r,r\geq 3,$ on
  $ \pi_\ast(\Omega \Sigma Y)\otimes \mathbb{Q} $  may be non-trivial (cf. Example \ref{one} below).

\begin{theorem}
If there is an isomorphism $ \pi_*(\Omega \Sigma Y) \otimes \mathbb{Q}\approx \pi_*(\Omega \Sigma Y')\otimes \mathbb{Q}$ of the
 $L_{\infty}$-bialgebras,
then $Y$ and $Y'$ are  equivalent in  rational homotopy category.
\end{theorem}
\begin{proof}

The isomorphism of the   $L_{\infty}$-bialgebras of the  theorem  implies
  the isomorphism \[(H_*(Y),\{\Delta_r\})\approx (H_*(Y'),\{\Delta'_r\})\] of $C_{\infty}$-coalgebras.  On the other hand, the $C_{\infty}$-coalgebra structure   of   $H_*(Y)$ uniquely characterizes $Y$ in the rational homotopy category \cite{kade},  so the proof of the theorem follows.

\end{proof}

\begin{example}\label{one}
1. Let a space $Y=S^2\vee S^2\vee S^2\cup_{f} e^5$ be obtained from the wedge of three
$2$-spheres by attaching the $5$-cell $e^5$ via
a map
$f:S^4\rightarrow S^2\vee S^2\vee S^2$ being a representative of the element
$ [i_1,[i_2,i_3]]\in \pi_4(S^2\vee S^2\vee S^2),$ the iterated Whitehead product
where
 $i_j:S^2\rightarrow S^2\vee S^2\vee S^2$
denotes the standard inclusion  at $j^{th}$-component $j=1,2,3.$  Then for $H:=H_\ast(Y)$ we have $H_0=\mathbb{Q},$
$H_2=\mathbb{Q}\oplus \mathbb{Q}\oplus \mathbb{Q}$ and $H_5=\mathbb{Q}.$
Although  $PH=H,$ the $C_\infty$-coalgebra structure on $H$ is non-trivial: namely,
 there is a representative
$\Delta_3:H\rightarrow H\otimes H\otimes H$ with
 $\Delta_3(w_5)=x_2\otimes y_2\otimes z_2$ for  $(x_2,y_2,z_2)\in H_2$ and $w_5\in H_5$ (compare \cite[Example 6.6]{hal-sta}). In particular, the $C_\infty$-coalgebra
 $(H, \Delta_3 )$ is primitive.
 We have
  $H_\ast(Y)\subset \pi_\ast(\Omega \Sigma Y)\otimes \mathbb{Q}$ with identifications
 $H_2(Y)=\pi_2(\Omega \Sigma Y) \otimes \mathbb{Q} $   and $H_5(Y)=\pi_5(\Omega \Sigma Y)\otimes \mathbb{Q}.$
 Consequently, for $\ell^3=\Delta^{sym}_3$ on $H,$ the cooperation
 \[\ell^3: \pi_5(\Omega \Sigma Y)\otimes \mathbb{Q}\rightarrow
 (\pi_2(\Omega \Sigma Y)\otimes \mathbb{Q})\otimes( \pi_2(\Omega \Sigma Y)\otimes \mathbb{Q})\otimes (\pi_2(\Omega \Sigma Y)\otimes \mathbb{Q} )      \]
  defined for $w_5\in \pi_5(\Omega \Sigma Y)\otimes \mathbb{Q}$
 by
  \[\ell^3(w_5)=x_2\otimes y_2\otimes z_2 - y_2\otimes x_2\otimes z_2+
 y_2\otimes z_2\otimes x_2-   x_2\otimes z_2\otimes y_2 +z_2\otimes x_2\otimes y_2- z_2\otimes y_2\otimes x_2\]
is non-trivial.

2. Let $Y'=S^2\vee S^2\vee S^2\vee S^5.$ Since $Y'$ is a suspension, the $C_\infty$-coalgebra structure on  $H_\ast(Y')$   is degenerated.
Hence, $H_\ast(Y)$  and $H_\ast(Y')$ are isomorphic as coalgebras but not as $C_\infty$-coalgebras. Consequently,
$ \pi_\ast(\Omega \Sigma Y)\otimes \mathbb{Q}$ and  $\pi_\ast(\Omega \Sigma Y')\otimes \mathbb{Q}$
are isomorphic as Lie algebras but not as  $L_\infty$-bialgebras.
In fact,   there is   only two rational homotopy types determined by these spaces in question.

\end{example}

Finally, remark that the above method applies to introduce an  $L_\infty$-bialgebra structure on the homotopy groups $\pi_\ast(\Omega \Sigma Y)$ whenever the Hurewitz homomorphism
$\pi_\ast(\Omega \Sigma Y)\rightarrow H_\ast(\Omega \Sigma Y;\mathbb{Z}) $
is an inclusion.


\begin{thebibliography}{99}

\bibitem{etin-schiff} P. Etingof and O. Schiffmann, Lectures on quantum groups, International Press (2002).


\bibitem{fernandez} J.M.M. Fern\'andez,
The Milnor-Moore theorem for $L_\infty$-algebras in rational homotopy theory, Math. Z., 300 (2022), 2147--2165.

\bibitem{hal-sta} S. Halperin and J.  Stasheff, Obstructions to homotopy
equivalences, Adv. in Math., 32 (1979), 233--279.


\bibitem{kade}  T. Kadeishvili, Cohomology $C_\infty$-algebra and rational homotopy type. In: Algebraic topology--old and new, Banach Center Publ., Polish Acad. Sci. Inst. Math., Warsaw, 85 (2009), 225--240.

\bibitem{lada-sta} T. Lada and J. Stasheff, Introduction to sh Lie algebras for physicits, International Journal of Theoretical Psysics, 32, no. 7 (1993), 1087--1103.

\bibitem{lada-markl} T. Lada and M. Markl, Strongly homotopy Lie algebras, Commun. Algebra, 23, no. 6 (1995), 2147--2161.

\bibitem {SU} S. Saneblidze and R. Umble,
Framed Matrices and $A_{\infty}$-Bialgebras, ASETMJ,   15, no. 4 (2022), 41--140.

\bibitem {SU2} S. Saneblidze and R. Umble, Comparing Diagonals on the Associahedra,  preprint,  math. AT/2207.08543,  to appear  in  HHA.


\end{thebibliography}
\end{document}